\documentclass[a4paper,12pt]{article}
\usepackage{amsmath,amsfonts,amssymb,amscd}
\usepackage{indentfirst,graphicx,epsfig}
\usepackage{graphicx,psfrag}
\usepackage{amsthm, amsfonts, color}
\usepackage[bookmarksnumbered, plainpages]{hyperref}
\usepackage[numbers,sort&compress]{natbib}
\input{epsf}

\title{On graphs with maximum Harary spectral radius\thanks{Supported by NSFC No.11371205 and PCSIRT.} }
\author{\small{Fei Huang, Xueliang Li, Shujing Wang}\\
{\small  Center for Combinatorics and LPMC-TJKLC}\\
{\small Nankai University, Tianjin 300071, China}\\\makeatletter
{\small Email: huangfei06@126.com; lxl@nankai.edu.cn;  \newcommand\figcaption{\def\@captype{figure}\caption}
wang06021@126.com} }  \newcommand\tabcaption{\def\@captype{table}\caption}
\date{}\makeatother
\newtheorem{theorem}{Theorem}[section]
\newtheorem{lemma}[theorem]{Lemma}
\newtheorem{coro}[theorem]{Corollary}

\begin{document}

\maketitle

\begin{abstract}
Let $G$ be a simple graph with vertex set $V(G) = \{v_1 ,v_2 ,\cdots ,v_n\}$.
The Harary matrix $RD(G)$ of $G$, which is initially called the reciprocal distance matrix,
is an $n \times n$ matrix whose $(i,j)$-entry is equal to $\frac{1}{d_{ij}}$
if $i\not=j$ and $0$ otherwise, where $d_{ij}$ is the distance of $v_i$ and $v_j$ in $G$.
In this paper, we characterize graphs with maximum spectral radius of Harary matrix in
three classes of simple connected graphs with $n$ vertices: graphs with
fixed matching number, bipartite graphs with fixed matching number,
and  graphs with given number of cut edges, respectively.\\[2mm]


\end{abstract}

\section{Introduction}

In this paper we are concerned with simple finite graphs. Undefined notation and terminology can be found in \cite{Bondy}. Let $G$ be a simple graph with vertex set $V(G) = \{v_1, v_2,\ldots, v_n\}$ and
edge set $E(G)$. Let $N_G(v)$ be the neighborhood of the vertex $v$ of $G$, and $d_{ij} $ be  the distance (i.e., the number of edges of a shortest path) between the vertices $v_i$ and $v_j$ in $G$.

The Harary
matrix $RD(G)$ of $G$, which is initially called the reciprocal distance matrix, is an $n \times n$ matrix $(RD_{ij} )$ such that
\begin{equation*}
RD_{ij}=
\left\{
  \begin{array}{ll}
   \frac{1}{d_{ij}} & \hbox{ if $i \neq j$,} \\
    0 & \hbox{otherwise}.
  \end{array}
\right.
\end{equation*}

As we know, in many instances the distant atoms influence each other much less than near atoms.  Harary matrix was introduced by Ivanciuc et al.\cite{Ivan} as an important molecular matrix to research this interaction, and it  was also  successfully used in a study concerning computer generation of acyclic graphs based on local vertex invariants and topological indices. Moreover, It is shown that the Harary spectral radius is able to produce fair QSPR models for the boiling points, molar heat capacities, vaporization enthalpies, refractive indices
and densities for $C_6$-$C_{10}$ alkanes.

Ivanciuc et al. \cite{Ivan1}
proposes to use the maximum eigenvalues of distance-based matrices as structural descriptors.
The lower and upper bounds of the maximum eigenvalues of Harary matrix, and the Nordhaus-Gaddum-type results for it were obtained in \cite{Zhou, Das}. Mathematical properties and applications of Harary index are reported in \cite{1Das,1Estrada,1Feng,1Xu,2Xu,1Zhou,2Zhou}. Some lower and upper bounds for Harary energy of
connected $(n,m)$-graphs were obtain in \cite{Gungor}.

A matching in a graph is a set of pairwise nonadjacent edges.
A maximum matching is one which covers as many vertices as possible.
The number of edges in a maximum matching of a graph $G$ is
called the matching number of $G$ and denoted by $\alpha'(G)$.
In this paper we characterize graphs with maximum spectral radius of Harary matrix in three classes of simple connected graphs with $n$ vertices: graphs with fixed matching number, bipartite graphs with fixed matching number, and  graphs with given number of cut edges, respectively.

\section{Preliminaries}
Since $RD$ is a real symmetric matrix, its eigenvalues are all real. Let
 $\rho(G)$ be the spectral radius of $RD(G)$, called Harary spectral radius. By the Perron-Frobenius theorem, the Harary spectral radius of a connected graph $G$ corresponds  to a unique positive unit eigenvector $X=(x_1, x_2, \cdots ,x_n)^T$,
  called principal eigenvector of $RD(G)$.
  Then
  \begin{equation}\label{1}
    \rho(G) x_i=\sum_{j\neq i}\frac{1}{d_{ij}}x_j.
  \end{equation}

The  following lemma is an immediate consequence of Perron-Frobenius Theorem.

\begin{lemma}\label{lem1}
Let $G$ be a connected graph with $u, v \in  V (G)$  and
$uv \notin  E(G)$. Then $\rho(G)<\rho(G + uv)$.
\end{lemma}

Let $G$ be a connected graph, and $H$ a subgraph of $G$. We know that
$H$ can be obtained from $G$ by deleting edges, and possibly vertices.

\begin{coro}\label{co1}
 If $H$ is a subgraph of a connected graph $G$, then $\rho(H)<\rho(G)$.
\end{coro}

  \begin{lemma}\label{lem2}
  Let $G$ be a connected graph with $v_r,v_s\in V(G)$.
  If $N_G(v_r)\setminus \{v_s\} = N_G(v_s)\setminus \{v_r\}$, then $x_r=x_s$.
  \end{lemma}

  \begin{proof}
  From Eq. \eqref{1}, we know that
  $$\rho(G) x_r=\sum_{j\neq  r}\frac{1}{d_{rj}}x_j= \frac{1}{d_{rs}}x_s+\sum_{j\neq s,r}\frac{1}{d_{rj}}x_j, $$

  and,
   $$\rho(G) x_s=\sum_{j\neq
  s}\frac{1}{d_{sj}}x_j= \frac{1}{d_{sr}}x_r+\sum_{j\neq
  s,r}\frac{1}{d_{sj}}x_j .$$
Since $N_G(v_r)\setminus \{v_s\} = N_G(v_s)\setminus \{v_r\}$,
  we have that  $d_{rj}=d_{sj}$ for $j\neq s,r$.
  Then $$\rho(G)( x_r-x_s)=-\frac{1}{d_{sr}} ( x_r-x_s),$$
  and thus $x_r=x_s$.
    \end{proof}

 \section{Graphs with given matching number}

Let $G_1\cup \cdots \cup G_k$  be the vertex-disjoint union of the graphs $G_1,\cdots, G_k$  $(k\geq 2)$,  and $G_1\vee G_2$ be the
graph obtained from $G_1 \cup  G_2$ by joining each vertex of $G_1$
 to each vertex of $G_2$.

\begin{lemma}\label{lem3}
Let $G_1= K_s\vee (K_{n_1} \cup K_{n_2} \cup \cdots \cup K_{n_k})$ and $G_2= K_s\vee (K_{n_1-1} \cup K_{n_2+1} \cup \cdots \cup K_{n_k})$.
 If  $n_2 \geq n_1\geq 2$,  then $\rho(G_1)<\rho(G_2)$.
\end{lemma}

\begin{proof}

Let $\rho(G_1)$ be the Harary spectral radius of $G_1$ and $X$ the  corresponding  principal eigenvector.
 By Lemma \ref{lem2}, $X$ can be written  as
  $$X=(\underbrace{y_1,\cdots y_1}_{n_1}, \underbrace{y_2,\cdots y_2}_{n_2},\cdots, \underbrace{y_k,\cdots y_k}_{n_k},\underbrace{y_0,\cdots y_0}_{s}).$$

From Eq. \eqref{1},  we have
\begin{eqnarray*}
\rho(G_1)y_1&=&(n_1-1)y_1+\frac{1}{2}n_2y_2+\sum _{i=3}^k\frac{1}{2}n_iy_i+sy_0,\\
\rho(G_1)y_2 &=&\frac{1}{2}n_1y_1+(n_2-1)y_2+\sum _{i=3}^k\frac{1}{2}n_iy_i+sy_0.
\end{eqnarray*}
It implies that $$\rho(G_1)(y_1-y_2)=\frac{1}{2}n_1y_1-y_1-\frac{1}{2}n_2y_2+y_2,$$
that is, $$(\rho(G_1)+1-\frac{1}{2}n_1)y_1=(\rho(G_1)+1-\frac{1}{2}n_2)y_2.$$
Note that $K_{s+n_2}$ is a subgraph of $G_1$ and $n_2\geq n_1$, we have that $$\rho(G_1)>\rho(K_{s+n_2})=s+n_2-1\geq n_2.$$
Then we have that $$y_1\leq y_2.$$

 From the definition of Harary matrix, we know that
 \begin{equation*}
 RD(G_1)=\left(
           \begin{array}{ccccc}
             (J-I)_{n_1\times n_1} & \frac{1}{2}J_{n_1\times n_2} & \cdots & \frac{1}{2}J_{n_1\times n_k} & J_{n_1\times s} \\
             \frac{1}{2}J_{n_2\times n_1} &(J-I)_{n_2\times n_2} & \cdots & \frac{1}{2}J_{n_2\times n_k} & J_{n_2\times s} \\
             \vdots & \vdots & \ddots & \vdots & \vdots \\
            \frac{1}{2}J_{n_k\times n_1} &J_{n_k\times n_2} & \cdots & (J-I)_{n_k\times n_k} & J_{n_k\times s} \\
             J_{s\times n_1} & J_{s\times n_2} & \cdots & J_{s\times n_k} & (J-I)_{s\times s}\\
           \end{array}
         \right),
\end{equation*}
and, 
\begin{equation*}
RD(G_2)=\left(
           \begin{array}{ccccc}
             (J-I)_{(n_1-1)\times (n_1-1)} & \frac{1}{2}J_{(n_1-1)\times (n_2+1)} & \cdots & \frac{1}{2}J_{(n_1-1)\times n_k} & J_{(n_1-1)\times s} \\
             \frac{1}{2}J_{(n_2+1)\times (n_1-1)} &(J-I)_{(n_2+1)\times (n_2+1)} & \cdots & \frac{1}{2}J_{(n_2+1)\times n_k} & J_{n_2\times s} \\
             \vdots & \vdots & \ddots & \vdots & \vdots \\
            \frac{1}{2}J_{n_k\times(n_1-1)} &J_{n_k\times (n_2+1)} & \cdots & (J-I)_{n_k\times n_k} & J_{n_k\times s} \\
             J_{s\times (n_1-1)} & J_{s\times (n_2+1)} & \cdots & J_{s\times n_k} & (J-I)_{s\times s}\\
           \end{array}
         \right).
\end{equation*}
Thus
\begin{equation*}
RD(G_2)- RD(G_1)=\left(
           \begin{array}{cccc}
             0_{(n_1-1)\times (n_1-1)} & -\frac{1}{2}J_{(n_1-1)\times 1} & 0_{(n_1-1)\times n_2} & 0 \\
             -\frac{1}{2}J_{1\times (n_1-1)} &0_{1\times 1} & \frac{1}{2}J_{1\times n_2}  & 0\\
            0_{n_2\times (n_1-1)} &\frac{1}{2}J_{n_2\times 1} & 0_{n_2\times n_2} & 0 \\
             0 & 0 & 0 & 0\\
           \end{array}
         \right).
  \end{equation*}
Hence
\begin{eqnarray*}
  \rho(G_2)-\rho(G_1) &\geq & X^T RD(G_2)X- X^TRD(G_1))X \\
  &=& X^T (RD(G_2)- RD(G_1))X=n_2y_1y_2-(n_1-1)y_1^2\\
  &>& 0.
\end{eqnarray*}
We complete the proof.
\end{proof}

\begin{lemma}\label{lem4}
Let $G= K_s\vee(\overline{K_{k-1}}\cup K_{2t+1})$ with $t\geq 1, k\geq 3$, and $G'= K_{s+t}\vee \overline{K_{k+t}}$. One has that $\rho(G)<\rho(G')$.
\end{lemma}
\begin{proof}
Let $\rho=\rho(G)$ be the Harary spectral radius of $G$ and $X$ be the principal eigenvector. By Lemma 2.3, $X$ is positive and can be written as
$$
X=(\underbrace{x,\cdots,x}_{s},\underbrace{y,\cdots,y}_{k-1},\underbrace{z,\cdots,z}_{2t+1})^T.
$$
From the definition of Harary matrix, we know that
\begin{equation*}
 RD(G)=\left(
           \begin{array}{cccc}
             (J-I)_{s\times s} & J_{s\times (k-1)} & J_{s\times (2t+1)} \\
             J_{(k-1)\times s} & \frac{1}{2}(J-I)_{(k-1)\times (k-1)}& \frac{1}{2}J_{(k-1)\times (2t+1)} \\
            J_{(2t+1)\times s} & \frac{1}{2}J_{(2t+1)\times (k-1)} & (J-I)_{(2t+1)\times (2t+1)} \\
           \end{array}
         \right)
\end{equation*}
and
\begin{equation*}
 RD(G')=\left(
          \begin{array}{ccccc}
            (J-I)_{s\times s} & J_{s\times (k-1)} & J_{s\times t} & J_{s\times t} & J_{s\times 1} \\
            J_{(k-1)\times s} & \frac{1}{2}(J-I)_{(k-1)\times (k-1)}& J_{(k-1)\times t} & \frac{1}{2}J_{(k-1)\times t} & \frac{1}{2}J_{(k-1)\times 1} \\
            J_{t\times s} &  J_{t\times (k-1)} &  (J-I)_{t\times t} &  J_{t\times t} &  J_{t\times 1} \\
            J_{t\times s} &  \frac{1}{2}J_{t\times (k-1)} &  J_{t\times t} &  \frac{1}{2}(J-I)_{t\times t} &  \frac{1}{2}J_{t\times 1} \\
            J_{1\times s} &  \frac{1}{2}J_{1\times (k-1)} &  J_{1\times t} &  \frac{1}{2}J_{1\times t} &  0_{1\times 1} \\
          \end{array}
        \right),
\end{equation*}
 thus
\begin{equation}\label{6}
  \begin{split}
  \rho(G')-\rho &\geq X^T(RD(G')-RD(G))X \\
                    &=t(k-1)yz-tz^2 \\
        &=tz((k-1)y-z).
  \end{split}
\end{equation}
As $X$ is the principal eigenvector corresponding to $\rho=\rho(G)$, from Eq. \eqref{1}, we have
\begin{eqnarray*}
  \rho y &=& sx+ \frac{k-2}{2}y+ \frac{1}{2}(2t+1)z,\\
  \rho z &=& sx+ \frac{k-1}{2}y+ 2tz.
\end{eqnarray*}
Then
\begin{equation}\label{7}
  \frac{y}{z}=\frac{2\rho-2t+1}{2\rho+1}.
\end{equation}
Hence
\begin{eqnarray*}
  (k-1)y-z &=& (k-1)\frac{2\rho-2t+1}{2\rho+1}z- z\\
   &=& \frac{z}{2\rho+1}(2(k-2)\rho-(k-1)(2t-1)-1 \\
   &=& \frac{2(k-2)z}{2\rho+1}(\rho-\frac{(k-1)(2t-1)+1}{2(k-2)})\\
   &>& \frac{2(k-2)z}{2\rho+1}(\rho-2t).
\end{eqnarray*}
Note that $K_{s+2t+1}$ is a subgraph of $G$, by Corollary \ref{co1}, we have that
$$\rho>\rho(K_{s+2t+1})=s+2t> 2t.$$
Hence we have that
\[
 \rho(G')-\rho > tz\frac{2(k-2)z}{2\rho+1}(\rho-2t)>0.
\]
We complete the proof.
\end{proof}

 A component of a graph $G$ is said to be even (odd) if it has an even (odd)
number of vertices. We use  $o(G)$ to denote  the number of odd components of $G$. Let $G$  be a graph  on $n$ vertices with $\alpha'(G)=p$. By the Tutte-Berge formula,
$$n - 2p = \max\{o(G -X) -|X| : X \subset V (G)\}.$$

\begin{theorem}
Let $G$  be a graph  on $n$ vertices with $\alpha'(G)=p$ which has the maximum Harary spectral radius. Then we have that

\begin{enumerate}
  \item if $p=\lfloor \frac{n}{2}\rfloor$, then $G= K_n$;
  \item if $1 \leq p<\lfloor \frac{n}{2}\rfloor$, then  $G= K_p\vee \overline{K_{n-p}} $.
\end{enumerate}
\end{theorem}

\begin{proof}

The  first assertion is trivial, and so we only  need to prove the second assertion.
Let $X_0$ be a  vertex subset such that $n - 2p = o(G -X_0) -|X_0|$. For convenience, let $|X_0| = s$ and $o(G -X_0) = k$.
Then $n -2p = k - s$. Since $1 \leq p<\lfloor \frac{n}{2}\rfloor$, we know that $k - s\geq 2$. Hence $k\geq 3$.

  If $G - X_0$ has an even component, then by adding an
edge to $G$  between a vertex of an even component and a vertex of an odd component of $G-X_0$, we obtain a graph $G'$ with matching number $p$. From Lemma \ref{lem1}, we know that $\rho(G')>\rho(G)$, a contradiction to the assumption that $G$ has the maximum Harary spectral radius. So we know that all the components of $G-X_0$  are odd.  Let $G_1, G_2,\cdots, G_k$ be the  odd components of $G-X_0$. Similarly, $G_1,G_2,\cdots ,G_k$ and
the subgraph induced by $X_0$ are all complete, and every vertex of $G_i (i = 1,\cdots, k)$ is adjacent to every vertex in $X_0$. Thus  $G= K_s\vee (K_{n_1} \cup K_{n_2} \cup \cdots \cup K_{n_k})$, where $n_i=|V(G_i)|$ for $i=1,2,\cdots, t$.

  First, we claim that $G - X_0$ has at most one odd component whose number
of vertex is more than one. Assume without loss of generality that $n_2\geq n_1\geq 3$. Let  $G'= K_s\vee (K_{n_1-2} \cup K_{n_2+2} \cup \cdots \cup K_{n_k})$. We can easily checked that $\alpha(G')=p$.
  From Lemma \ref{lem3}, we know that $\rho(G)<\rho(G')$, a contradiction.
  Then $G= K_s\vee (\overline{K_{k-1}}\cup K_{t})$, where $s+t+k-1=n$. By Lemma \ref{lem4},  we know that $t=1$. The result follows.
\end{proof}

\section{Bipartite graphs with given matching number}
\begin{lemma}[\cite{Cui}]
Let $K_{n_1,n_2}$ be a completed bipartite graph with $n=n_1+n_2$ vertices. One has that
\[
\rho(K_{n_1,n_2})=\frac{1}{4}(n-2+\sqrt{n^2+12n_1n_2}).
\]
\end{lemma}
\begin{coro} \label{coro:2.1} \
\begin{equation}\label{9}
\rho(K_{1,n-1})<\rho(K_{2,n-2})<\ldots< \rho(K_{\lfloor\frac{n}{2}\rfloor,\lceil\frac{n}{2}\rceil}).
\end{equation}
\end{coro}

 A covering of a graph $G$ is a vertex subset
$K\subseteq V(G)$ such that each edge of $G$ has at least one end in
the set $K$. The number of vertices in a minimum covering of a graph
$G$ is called the covering number of $G$ and denoted by $\beta(G)$.

\begin{lemma}\label{lem3.3}   \emph{(The K\"{o}nig-Egerv\'{a}ry Theorem, \cite{Eger,K}).}
In any bipartite graph, the number of edges in a maximum matching is equal to
the number of vertices in a minimum covering.
\end{lemma}

Let $G=G[X,Y]\not= K_{p,n-p}$ be a bipartite graph such that $\alpha'(G)=p$. From Lemma \ref{lem3.3}, we know that $\beta(G)=p$. Let
$S$ be  a minimum covering of $G$ and  $X_1=S \cap X\not=\emptyset$, $Y_1=S \cap
Y\not=\emptyset$. Set $X_2=X\setminus X_1$, $Y_2=Y\setminus
Y_1$. We have that $E(X_2,Y_2)=\emptyset$ since $S$ is a covering of
$G$.

Let $G^*[X,Y]$ be a bipartite graph with the same vertex set as $G$
such that $E(G^*)=\{xy:x\in X_1, y\in Y\}\cup \{xy: x\in X_2, y\in
Y_1\}$. Obviously, $G$ is a subgraph of $G^*$. From Lemma \ref{lem1}, we
know that

\begin{equation}\label{3.1}
\rho(G)\leq \rho(G^*),
\end{equation}
with equality holds if and only if $G= G^*$.

Let $$G'=G^*-\{uv: u\in X_2, v\in Y_1\}+\{uw: u\in X_2, w\in X_1\}, $$ and
$$G''=G^*-\{uv: u\in X_1, v\in Y_2\}+\{uw: u\in Y_2, w\in Y_1\}. $$
Then we have the following conclusion:
\begin{figure}[h,t,b,p]
\begin{center}
\scalebox{1}[1]{\includegraphics{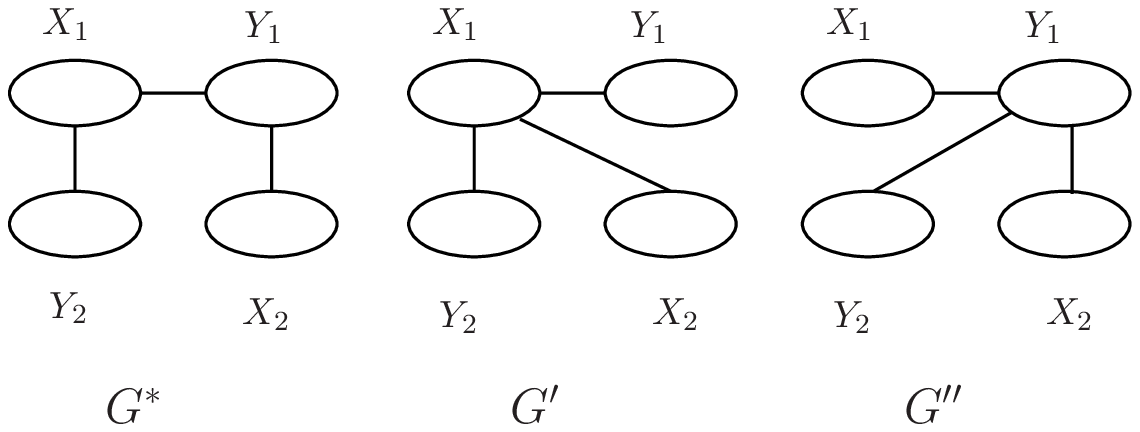}}\\[12pt]
Figure~1. $G^*, G'$ and $G''$
\end{center}
\end{figure}

\begin{lemma}\label{lem:3.2}
Let $G^*, G'$ and $G''$ be the graph defined above (see Figure 1) with $X_2\not=\emptyset$ and $Y_2\not=\emptyset$.
Then one has
\begin{equation}\label{1.2}
\rho(G^*)< \rho(G'), \ \ \textit{or } \ \  \rho(G^*)< \rho(G'').
\end{equation}
\end{lemma}

\begin{proof}
Let $\rho=\rho(G^*)$ be the Harary spectral radius of $G^*$ and $X$ the principal eigenvector. By Lemma 2.3, $X$ is positive and can be written as
$$
X=(\underbrace{x_1,\cdots,x_1}_{a},\underbrace{x_2,\cdots,x_2}_{b},
\underbrace{y_1,\cdots,y_1}_{c},\underbrace{y_2\cdots,y_2}_{d})^T,
$$
where $a=|X_1|, b=|X_2|, c=|Y_1|$ and $d=|Y_2|$.

As
\begin{equation*}
 RD(G^*)=\frac{1}{2}(J-I)+\left(
           \begin{array}{cccc}
             0_{a\times a} &  \frac{1}{2}J_{a\times b} & J_{a\times c} & J_{a\times d} \\
             \frac{1}{2}J_{b\times a} & 0_{b\times b} & J_{b\times c} & \frac{1}{3}J_{b\times d} \\
             J_{c\times a} & J_{c\times b} & 0_{c\times c} & \frac{1}{2}J_{c\times d} \\
             J_{d\times a} & \frac{1}{3}J_{d\times b} & \frac{1}{2}J_{d\times c} & 0_{d\times d} \\
           \end{array}
         \right)
\end{equation*}
and
\begin{equation*}
 RD(G')=\frac{1}{2}(J-I)+\left(
           \begin{array}{cccc}
             0_{a\times a} &  J_{a\times b} & J_{a\times c} & J_{a\times d} \\
             J_{b\times a} & 0_{b\times b} & \frac{1}{2}J_{b\times c} & \frac{1}{2}J_{b\times d} \\
             J_{c\times a} & \frac{1}{2}J_{c\times b} & 0_{c\times c} & \frac{1}{2}J_{c\times d} \\
             J_{d\times a} & \frac{1}{2}J_{d\times b} & \frac{1}{2}J_{d\times c} & 0_{d\times d} \\
           \end{array}
         \right),
\end{equation*}
we have

\begin{equation}\label{4}
  \begin{split}
    &X^T(RD(G')-RD(G^*))X \\
     &= X^T\left(
            \begin{array}{cccc}
             0_{a\times a} &  \frac{1}{2}J_{a\times b} & 0_{a\times c} & 0_{a\times d} \\
             \frac{1}{2}J_{b\times a} & 0_{b\times b} & -\frac{1}{2}J_{b\times c} & \frac{1}{6}J_{b\times d} \\
             0_{c\times a} & -\frac{1}{2}J_{c\times b} & 0_{c\times c} & 0_{c\times d} \\
             0_{d\times a} & \frac{1}{6}J_{d\times b} & 0_{d\times c} & 0_{d\times d} \\
            \end{array}
          \right)
    X \\
    &=abx_1x_2-bcx_2y_1+\frac{1}{3}bdx_2y_2\\
    &=bx_2(ax_1-cy_1)+ \frac{1}{3}bdx_2y_2.
  \end{split}
\end{equation}
Similarly, one has that

$$X^T(RD(G'')-RD(G^*))X =dy_2(cy_1- ax_1)+ \frac{1}{3}bdx_2y_2.$$
  It is easy to see that either $X^T(RD(G')-RD(G^*))X>0$ or $X^T(RD(G'')-RD(G^*))X>0,$ i.e., $\rho(G^*)<\rho(G')$ or $\rho(G^*)<\rho(G'')$.
This completes the proof.
\end{proof}
By (\ref{3.1}) and (\ref{1.2}), together with Corollary \ref{coro:2.1}, it is straightforward to see that
\begin{theorem}
For any bipartite graph $G$ with matching number $p$ and $G\not= K_{p,n-p}$, one has that $\rho(G)<\rho(K_{p,n-p})$.
\end{theorem}

\section{Graphs with given number of cut edges}

\begin{lemma}\label{lem5.1}
 Let $G$ be a graph with a cut edge $e=w_1w_2$,  and  $G'$ be the graph obtained from $G$ by contracting edge $e$  and  adding  a pendent edge attaching at the contracting vertex (see Figure 2). If $d_G(w_i)\geq 2$ for $i=1,2$,  we have that  $\rho(G')>\rho(G)$.
  \end{lemma}
\begin{figure}[h,t,b,p]
\begin{center}
\scalebox{0.9}[0.9]{\includegraphics{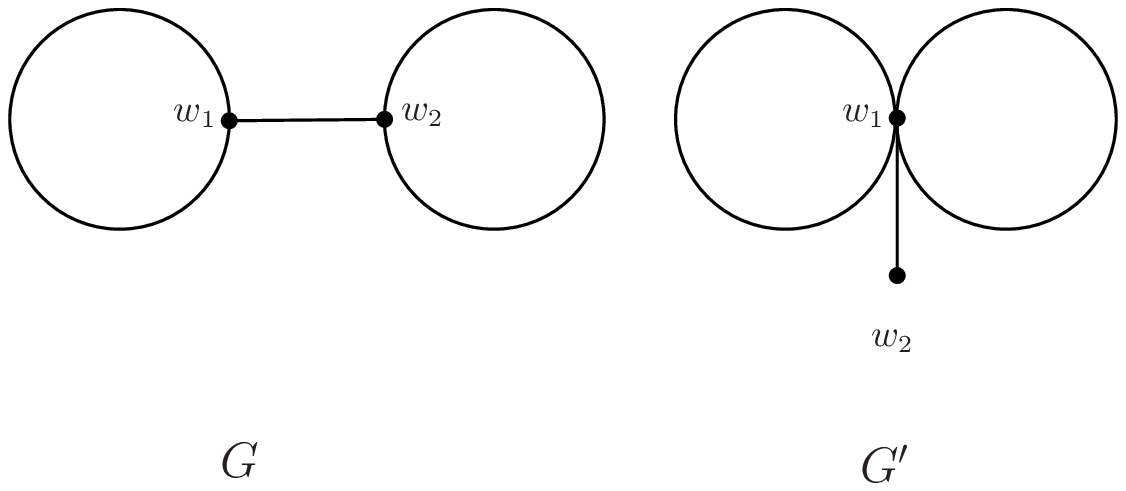}}\\[12pt]
Figure~2. $G$ and $G'$
\end{center}
\end{figure}
  \begin{proof}
 Let $\rho(G)$ be the Harary spectral radius of $G$ and $X$ the  corresponding  principal eigenvector. Without loss of generality, we assume that $x_{w_1}\geq x_{w_2}$. We denote the contracting vertex by $w_1$, and the pendant edge by $w_1w_2$.
 Let $G_i$ be the component of $G-e$ that contains $w_i$ for $i=1,2$. Let $V_1'=V(G_1)\setminus\{w_1\}$ and $V_2'=V(G_2)\setminus\{w_2\}$.
 For any two vertices $u$ and $v$, we have that

 \begin{equation*}
 d_{G'}(u,v)=
\left\{
  \begin{array}{ll}
    d_{G}(u,v)-1, & \hbox{if $u\in V(G_1)$ and $v\in V_2'$,} \\
    d_{G}(u,v)+1,& \hbox{if $u=w_2$ and $v\in V_2'$},\\
   d_G(u,v), & \hbox{otherwise.}
  \end{array}
\right.
\end{equation*}

 Let
 $$
 A=\sum_{w_1\in V_1', w_2\in V_2'}(\frac{1}{d_{G'}(w_1,w_2)}-\frac{1}{d_{G}(w_1,w_2)})x_{w_1}x_{w_2}>0.
 $$

 From the definition of Harary matrix, we know that
\begin{eqnarray*}
  \rho(G')-\rho(G) &\geq & X^T RD(G')X- X^TRD(G))X \\
  &=&\sum_{u,v \in V(G)}(\frac{1}{d_{G'}(u,v)}-\frac{1}{d_{G}(u,v)})x_{u}x_{v}\\
  &=& 2A+2\sum_{u=w_1,v\in  V_2'}(\frac{1}{d_{G'}(u,v)}-\frac{1}{d_{G}(u,v)})x_{u}x_{v}\\
  &&+2\sum_{u=w_2, v\in V_2'}(\frac{1}{d_{G'}(u,v)}-\frac{1}{d_{G}(u,v)})x_{u}x_{v}\\
  &=& 2A+2 x_{v}\sum_{v\in  V_2'}(\frac{x_{w_1}}{d_{G}(w_1,v)(d_{G}(w_1,v)-1)} -\frac{x_{w_2}}{d_{G}(w_2,v)(d_{G}(w_2,v)+1)})\\
  &=& 2A+2(x_{w_1}-x_{w_2})x_{v}\sum_{v\in  V_2'}\frac{1}{d_{G}(w_1,v)(d_{G}(w_1,v)-1)}\\
  &\geq& 2A>0.
\end{eqnarray*}
Note that the last equality holds since $d_{G}(w_1,v)=d_{G}(w_2,v)+1$ for any   $v\in  V_2'$.
Hence we have our conclusion.
\end{proof}

Assume that $r_1, r_2,\cdots, r_s$  are positive integers, and $s\leq t$. Let $K_t(r_1, r_2,\cdots, r_s)$  be the graph that is obtained from $K_t$ with $V(K_t)=\{v_1,v_2,\cdots,v_t\}$  by  attaching $r_i$ pendant edges to vertex  $v_i$ for $1\leq i\leq s$.

\begin{lemma}\label{lem5.2}
 Let $G=K_t(r_1, r_2,\cdots, r_s)$ and $G'=K_t(r_1+ r_2+\cdots+ r_s)$. Then  $\rho(G')>\rho(G)$.
  \end{lemma}

\begin{proof}
 Let $\rho(G)$ be the Harary spectral radius of $G$ and $X$ the  corresponding  principal eigenvector. Let $R_i$ be set of pendant vertices that is adjacent to $v_i$ in $G$.  From Lemma \ref{lem2}, we can suppose that  $x_u=a_i$ for all $u\in R_i$ ($1\leq i\leq s$).
  Without loss of generality, assume that $x_{v_1}\geq x_{v_i}$ for $2\leq i\leq s$.
Let $G''=G-\{v_2w:w\in R_2\}+\{v_1w:w\in R_2\}$, that is,  $G''=K_t(r_1+r_2, r_3,\cdots, r_s)$.
 For any two vertices $u$ and $v$, if neither $u$ nor $v$ belongs to $R_2$, we know that $d_G(u,v)=d_{G''}(u,v)$; If both $u$ and  $v$ belong to $R_2$, we can also get $d_G(u,v)=d_{G''}(u,v)$. If exactly one of $u$ and  $v$ belongs to $R_2$, say $u\in R_2$, we have the following equation.

 \begin{equation*}
d_{G''}(u,v)=
\left\{
  \begin{array}{ll}
    d_{G}(u,v)-1=2, & \hbox{if $v\in R_1$,} \\
    d_{G}(u,v)-1=1,& \hbox{if $v=v_1$}\\
    d_{G}(u,v)+1=2, & \hbox{if $v=v_2$,} \\
   d_G(u,v), & \hbox{otherwise.}
  \end{array}
\right.
\end{equation*}

 From the definition of Harary matrix, we know that
\begin{eqnarray*}
  \rho(G'')-\rho(G) &\geq & X^T RD(G'')X- X^TRD(G))X \\
  &=&\sum_{u,v \in V(G)}(\frac{1}{d_{G''}(u,v)}-\frac{1}{d_{G}(u,v)})x_{u}x_{v}\\
  &=&2 \sum_{u\in R_2,v \notin R_2}(\frac{1}{d_{G''}(u,v)}-\frac{1}{d_{G}(u,v)})x_{u}x_{v}\\
  &=& 2r_2a_2\big(\sum_{v\in R_1}(\frac{1}{2}-\frac{1}{3})x_{v}+
 (1-\frac{1}{2})x_{v_1}+(\frac{1}{2}-1)x_{v_2}\big)\\
 &=& \frac{1}{3}r_1r_2a_1a_2+
  r_2a_2(x_{v_1}-x_{v_2})\\
 &>&0.
\end{eqnarray*}
By repeating this process until all the pendant edges have a common end, we can obtain our conclusion.
\end{proof}

From Lemma \ref{lem1}, Lemma \ref{lem5.1} and Lemma \ref{lem5.2}, we have the following theorem.

\begin{theorem}
Let $G$  be a graph  on $n$ vertices with $p$ cut edges which has the maximum Harary spectral radius, then $G=K_{n-p}(p)$.
\end{theorem}

\begin{coro}
The $n$-vertex star $S_n$ is the unique tree on $n$ vertices which has the maximum Harary spectral radius.
\end{coro}

\end{document}